\documentclass{amsart}

\usepackage{amssymb}
\usepackage[all]{xy}
\usepackage{hyperref}
\usepackage[dvipsnames,usenames]{xcolor}
\hypersetup{
    colorlinks=true,
    linkcolor={red!80!black},
    citecolor={green!50!black},
    urlcolor={red!80!black}
}

\usepackage{enumitem}   

\setlist[enumerate]{itemsep=.2em,topsep=.2em,leftmargin=1.25em,itemindent=2.0em}

\newtheorem{thm}{Theorem}

\newtheorem{lem}[thm]{Lemma}
\newtheorem{cor}[thm]{Corollary}

\newtheorem{prop}[thm]{Proposition}

\newtheorem{conj}[thm]{Conjecture}
   
\theoremstyle{definition}
\newtheorem{defn}[thm]{Definition}

\newtheorem{say}[thm]{}

\newtheorem{ques}[thm]{Question}    

\newtheorem*{ack}{Acknowledgments}      
\newtheorem{notation}[thm]{Notation}   
  
\newtheorem{defn-thm}[thm]{Definition--Theorem}  
\newtheorem{defn-lem}[thm]{Definition--Lemma}  

\theoremstyle{remark}

\newtheorem{subrem}[equation]{Remark}
\newtheorem{subex}[equation]{Example}
\numberwithin{equation}{thm}
\renewcommand{\c}[0]{{\mathbb C}}  

\renewcommand{\o}[0]{{\mathcal O}} 
\newcommand{\z}[0]{{\mathbb Z}}

\renewcommand{\r}[0]{{\mathbb R}} 

\renewcommand{\a}[0]{{\mathbb A}}

\newcommand{\q}[0]{{\mathbb Q}}

\newcommand{\supp}[0]{\operatorname{Supp}}    
\newcommand{\red}[0]{\operatorname{red}}

\newcommand{\diff}[0]{\operatorname{Diff}}

\newcommand{\rdown}[1]{\lfloor{#1}\rfloor}

\newcommand{\onto}[0]{\twoheadrightarrow}

\newcommand{\coeff}[0]{\operatorname{coeff}} 
\newcommand{\mld}[0]{\operatorname{mld}}

\def\into{\DOTSB\lhook\joinrel\to}

\def\qis{\,{\simeq}_{qis}\,}

\def\loccoh#1.#2.#3.#4.{H^{#1}_{#2}(#3,#4)}

\DeclareMathAlphabet{\mathchanc}{OT1}{pzc}%
                                {m}{it}

\usepackage[all]{xy}\xyoption{dvips}

\newcommand{\lcg}[0]{\operatorname{lcg}}

\begin{document}
\bibliographystyle{amsalpha}

 \hfill\today

\title{Du~Bois property of log centers}
\author{J\'anos Koll\'ar and S\'andor J Kov\'acs}
\address{JK: Department of Mathematics, Princeton University, Fine
  Hall, Washington Road, Princeton, NJ 08544-1000, USA} 
\email{kollar@math.princeton.edu}
\urladdr{http://www.math.princeton.edu/$\sim$kollar}
\address{SK: University of Washington, Department of Mathematics, Box 354350,
Seattle, WA 98195-4350, USA}
\email{skovacs@uw.edu}
\urladdr{http://www.math.washington.edu/$\sim$kovacs}

\maketitle

The analytic aspects of multiplier ideals, log canonical thresholds and log canonical
centers played an important role in several papers of Demailly, including
\cite{MR1786484, MR1919457, dk, MR2978333, MR3179606,MR3525916, MR3647124,
  MR3923220}.

Log canonical centers are seminormal by \cite{ambro, fuj-book}, even Du~Bois by
\cite{k-db, k-db2}.  This has important applications to birational geometry and
moduli theory; see \cite{k-db, k-db2} or \cite[Sec.2.5]{k-modbook}.

We recall the concept of Du~Bois singularities in
Definition--Theorem~\ref{schwede.thm}.  An unusual aspect is that this notion makes
sense for complex spaces that have irreducible components of different
dimension. This is crucial even for the statement of our theorem.

In this note we generalize the results of \cite{k-db, k-db2} by showing that if a
closed subset $V\subset X$ is `close enough' to being a union of log canonical
centers, then it is Du~Bois.

The minimal log discrepancy---denoted by $\mld(V, X, \Delta)$ as in
Definition~\ref{crep.log.str.mld.defn}---is a nonnegative rational number, that
measures the deviation from being a union of log canonical centers. The log canonical
gap in dimension $n$---denoted by $\lcg(n)$ as in
Definition~\ref{epsilon.defn}---gives the precise notion of `closeness.'

\begin{thm}\label{DB.epsilon.thm}
  Let $(X, \Delta)$ be a log canonical pair of dimension $n$, and $V\subset X$ a
  closed subset such that $\mld(V, X,\Delta)<\lcg(n)$.  Then $V$ has Du~Bois
  singularities.
\end{thm}

The theorem applies to algebraic varieties and algebraic spaces of finite type over a
field of characteristic 0.

By \cite{k-1/6}, if $\mld(V, X,\Delta)<\tfrac16$, then $V$ is seminormal, and
$\tfrac16$ is optimal in every dimension $\geq 2$.  The bound $\lcg(n)$ is also
optimal for every $n$, but its value is not known for $n\geq 4$, and $\lcg(n)$
converges to 0 very rapidly, see (\ref{epsilon.defn.3}) and (\ref{epsilon.defn.4}).

\medskip

We follow the terminology of \cite{km-book} and of \cite{kk-singbook}.

\begin{defn}[Minimal log discrepancy]\label{crep.log.str.mld.defn}
  Let $(X,\Delta)$ be a log canonical pair and $V\subset X$ an irreducible
  subvariety. The {\it minimal log discrepancy} of $V$ is the infimum of the numbers
  $1+a(E,X,\Delta)$, where $E$ runs through all divisors over $X$ that dominate $V$,
  where $a(E,X,\Delta)$ denotes the discrepancy as in \cite[2.25]{km-book}.  It is
  denoted by $\mld(V,X,\Delta)$.  Thus if $V=D$ is a divisor on $X$, then
  $\mld(D,X,\Delta)=1-\coeff_D\Delta$.

  $V$ is a {\it log canonical center} or {\it lc center} of $(X,\Delta)$ iff
  $\mld(V,X,\Delta)=0$. It is sometimes convenient to view $X$ itself as a log
  canonical center.

  Let $V\subset Z$ be a closed subset with irreducible components $V_i$. We define
  its minimal log discrepancy as
  $\mld(V,X,\Delta):=\max_i \bigl\{\mld(V_i,X,\Delta)\bigr\}$.
\end{defn}

\begin{defn}[Log canonical gap]\label{epsilon.defn}
  The log canonical gap in dimension $n$, denoted by $\lcg(n)$, is the largest (real)
  number $\epsilon$ with the following property.
  \begin{enumerate}
  \item Let $(X, dD)$ be a $\q$-factorial, log canonical pair with $\dim X=n$ and $D$
    a $\z$-divisor. Assume that $d>1-\epsilon$. Then $(X, D)$ is also log canonical.
  \end{enumerate}
  Note that $1-\lcg(n)$ is the largest log canonical threshold in dimension $n$ (that
  is $<1$).  
  By replacing all $d_i$ by the smallest one we see that it is also the largest
  $\epsilon$ with the following 
  property.
  \begin{enumerate}[resume]
  \item Let $(X, \sum d_iD_i)$ be a $\q$-factorial, log canonical pair with
    $\dim X=n$, where $D_i$ are $\z$-divisors and $d_i\in\q$. Assume that
    $d_i>1-\epsilon$ for every $i$. Then $(X, \sum D_i)$ is also log canonical.
  \end{enumerate}
  
  It is easy to see that $\lcg(2)=\frac16$, and $\lcg(3)=\frac1{42}$ by
  \cite[5.5]{k-logsurf}.  A difficult theorem \cite[Thm.1.1]{hmx-acc} says that
  $\lcg(n)$ is positive for every $n$, but no explicit lower bound is known for
  $n\geq 4$.
  
  \setcounter{equation}{2}
  \begin{subrem}\label{epsilon.defn.3}
    Let $(X, dD)$ be a $\q$-factorial, log canonical pair such that $(X, D)$ is not
    log canonical.  As in \cite[2.49--53]{km-book}, it has a quasi-\'etale cover
    $\pi:(X', D')\to (X, D)$ such that $K_{X'}$ and $D'$ are both Cartier.  By Reid's
    lemma, $(X', dD')$ is log canonical and $(X', D')$ is not log canonical; see
    \cite[5.20]{km-book}.  Since $K_{D'}$ is Cartier, log canonical coincides with
    Du~Bois by \cite{Kovacs99}. This shows that the bound $\mld(V, X,\Delta)<\lcg(n)$
    is optimal in Theorem~\ref{DB.epsilon.thm}.
  \end{subrem}
  

\begin{subex}\label{epsilon.defn.4}
  Set $c_1=2$ and let $c_{k+1}:=c_1\cdots c_k+1$; it is called Euclid's or
  Sylvester's sequence, see \cite[A00058]{sequences}.  It starts as
  $2,3,7,43,1807, 3263443, ...$ 
  Then $D_n:=\bigl(z_1^{c_1}+\cdots+z_n^{c_n}=0\bigr)\subset
  \c^n$ is not log canonical, but $\bigl(\a^n, \bigl(1-\tfrac{1}{c_1\cdots
    c_{n}}\bigr)D\bigr)$ is log canonical; see \cite[8.6]{kk-singbook} for details.

  Thus $\lcg(n)\leq \tfrac{1}{c_1\cdots c_{n}}$, and the latter goes to 0 doubly
  exponentially.
\end{subex}
%
\end{defn}

\subsection*{Du~Bois property}{\ }

\noindent
Let $M$ be a compact K\"ahler manifold. One of the useful consequences of the Hodge
decomposition is the surjectivity of the natural map
\[
  H^i(M, \c)\onto H^i(M, \o_M).
\]
Roughly speaking, projective varieties with Du~Bois singularities form the largest
class that is stable under natural operations (small deformations, products, general
hyperplane sections) where the above surjectivity still holds.
For our curent purposes the Du~Bois property can be handled as a black box. We list
in Paragraph~\ref{db.sings.props.say} all the properties that we use.  We give
references to the original papers; \cite[Chap.6]{kk-singbook} is a suitable general
introduction.

The original and most useful definition is rather complicated; see the papers
\cite{DuBois81, MR2339829, Kovacs11d} or \cite[Sec.6.1]{kk-singbook}.  The following
characterization emphasizes that Du~Bois is a generalization of seminormality. We can
take (\ref{schwede.thm}.\ref{item:2}) as our definition. (We use `$X$ is Du~Bois' and
`$X$ has Du~Bois singularities' as synonyms.)  We state the version given in
\cite[6.4]{Kovacs-Schwede11}.

\begin{defn-thm}\cite{MR2339829}\label{schwede.thm}
  Let $X$ be reduced and $Y\supset X$ a smooth space containing it.  Let
  $\pi:Y'\to Y$ be an embedded log resolution of $X$, that is, $Y'$ is smooth and
  $E:=\red \pi^{-1}(X)$ is a simple normal crossing divisor.  Then $X$ is
  \begin{enumerate}
  \item\label{item:1} seminormal iff $\pi_*\o_E=\o_X$, and
  \item\label{item:2} Du~Bois iff $\pi_*\o_E=\o_X$ and $R^i\pi_*\o_E=0$ for
    $i>0$.\qed
  \end{enumerate}
\end{defn-thm}

In particular, if $X$ is Du~Bois, then it is reduced and seminormal.

If $X$ is smooth then the blow-up $Y':=B_XY\to Y$ shows that $X$ is Du~Bois.

\begin{say}[Properties of Du~Bois singularities that we
  use]\label{db.sings.props.say}
 
  We work either with algebraic spaces of finite type over a field of characteristic
  0. 
  Note that we allow them to have irreducible components of different dimensions.
  
  \medskip\noindent {\it Property~\thethm.0.}  Smooth implies
  Du~Bois.  Du~Bois implies reduced and seminormal.
  
  \medskip\noindent {\it Property~\thethm.1.} \cite{k-db, k-db2} Let $(X, \Delta)$ be
  an log canonical pair and $V\subset X$ a union of some of its log canonical
  centers. Then $V$ is Du~Bois. More generally, this holds for log canonical centers
  of crepant log structures, as in Definition~\ref{crep.l.s.def}.

  \medskip
  \noindent {\it Property~\thethm.2.} \cite[2.11-12]{MR2784747} Let
  $X_1, X_2\subset X$ be closed subspaces. If 3 of
  $\{X_1\cap X_2, X_1, X_2, X_1\cup X_2\}$ are Du~Bois, then so is the 4th.

  \medskip
  \noindent {\it Property~\thethm.3.} \cite[1.6]{k-db}, \cite[3.3]{Kovacs11c} and
  \cite[6.27]{kk-singbook}.  Let $p: Y\to X$ be a proper surjective morphism,
  $V\subset X$ a closed, reduced subscheme, and $D:= \supp p^{-1}(V)$. Assume that
  $\o_X(-V)\to R p_*\o_Y(-D)$ has a left inverse and $Y, D$ are Du~Bois.  Then $X$ is
  Du~Bois $\Leftrightarrow$ $V$ is Du~Bois.
\end{say}

In applications they key is to find examples where
Property~\ref{db.sings.props.say}.3 applies. The following gives most known cases.

\begin{thm}\label{1/6.sn.lem}
  Let $f:Y\to Z$ be a projective morphism with connected fibers between normal
  spaces. Assume that $(Y, \Delta)$ is $\q$-factorial, dlt and
  $K_Y+\Delta\sim_{f,\r}0$.  Let $D$ be an effective $\z$-divisor such that
  $\rdown{\Delta}\subset \supp D\subset \supp \Delta$ and $-D$ is $f$-semiample. Set
  $V=f(\supp D)$.  Then $\o_Z(-V)\to R f_*\o_Y(-D)$ has a left inverse
\end{thm}

Note that, since $-D$ is $f$-semiample, $D$ does not dominate $Z$. Thus
$V\subsetneq Z$ and $\o_Z(-V)$ makes sense.

\begin{proof} If  $f$ is birational then 
  the proof  is much simpler, and  worth doing separately.

  Choose $\epsilon>0$ such that $\Theta:=\Delta-\epsilon D$ is effective.  Note that
  \[
  -D\sim_{f,\r} K_Y+\Theta+(1-\epsilon)(-D),
  \eqno{(\ref{1/6.sn.lem}.1)}
  \]

  $(Y, \Theta)$ is klt and $(1-\epsilon)(-D)$ is $f$-semiample.

  In the birational case, the general form of Grauert-Riemenschneider vanishing gives
  that $R^if_*\o_Y(-D)=0$ for $i>0$; see \cite[2.68]{km-book}.
  Thus $R f_*\o_Y(-D)\qis f_*\o_Y(-D)=\o_Z(-V)$.

  If $D$ does not dominate $Z$, then the assumption $\rdown{\Delta}\subset \supp D$
  implies that the generic fiber is klt.  Also, $D=\supp f^{-1}(V)$, since $-D$ is
  $f$-nef and the fibers are connected. We can now use \cite[3.1]{k-hdi2}, more
  precisely the form given in \cite[10.41]{kk-singbook}, to get the required left
  inverse.
\end{proof}

\subsection*{The klt case   of Theorem~\ref{DB.epsilon.thm}}{\ }

\begin{say}
  \label{DB.epsilon.thm.klt.pf}
  The proof is short and follows \cite{k-1/6}.

  First we show the special case when $\supp V$ is a divisor; see
  Lemma~\ref{DB.epsilon.div.lem}.

  In general, we find a dlt modification $g:(Y, \Delta_Y)\to (X, \Delta)$ such that
  $D:=g^{-1}(V)$ is a divisor and $\mld(D, Y, \Delta_Y)=\mld(V, X, \Delta)$; see
  Proposition~\ref{thm:one-div-models}.  Choose $\epsilon>0$ such that
  $\Theta:=\Delta_Y-\epsilon D$ is effective.  Note that
  \[
    -D\sim_{g,\r} K_Y+\Theta+(1-\epsilon)(-D), \eqno{(\ref{DB.epsilon.thm.klt.pf}.1)}
  \]
  $(Y, \Theta)$ is klt.
 If $-D$ is  $g$-nef, then 
  Grauert-Riemenschneider vanishing applies to  $R^i g_*\o_Y(-D)$.
   We can achive these 
after running a suitable MMP; see 
 Lemma~\ref{mmp.on.crep.D.lem}.
Thus we may assume that 
 $\o_X(-V)\cong R g_*\o_Y(-D)$.
 $V$ is now 
 Du~Bois  by  (\ref{db.sings.props.say}.3). \qed
\end{say}

\begin{lem}  \label{DB.epsilon.div.lem}
  Theorem~\ref{DB.epsilon.thm} holds if $X$ is $\q$-factorial and 
  $V$ has pure codimension 1.
\end{lem}

\begin{proof} Write $V=\cup_{i\in I} D_i$ where the $D_i\subset X$ are irreducible divisors.  Note that $\mld(D, X, \Delta)=1-\coeff_{D}\Delta$ for any irreducible divisor $D$.
  Thus we can write  $\Delta=\sum_{i\in I} d_iD_i+\Delta'$ where $d_i>1-\lcg(n)$ and $D_i\not\subset\supp \Delta'$.
Since $X$ is $\q$-factorial, 
$(X, \sum_{i\in I}d_i D_i)$ is also lc, hence so is    $(X, \sum_{i\in I} D_i)$  by Definition~\ref{epsilon.defn}.
Note that each $D_i$ is a log canonical center of $(X, \sum_{i\in I} D_i)$,  so 
$\cup_iD_i$ is Du~Bois by  (\ref{db.sings.props.say}.1).
\end{proof}

\begin{prop} \cite[1.38]{kk-singbook}\label{thm:one-div-models} 
 Let  $(X,\Delta)$ be log canonical,  
and $\{E_i:i\in I\}$  finitely many exceptional divisors over $X$ such that
$-1\leq a(E_i,X,\Delta)< 0$. 
Then there is a $\q$-factorial, dlt 
modification  $g:\bigl(Y,\Delta_Y\bigr)\to (X,\Delta)$
such that
\begin{enumerate}
\item  the $\{E_i:i\in I\}$ are among the  exceptional divisors of $g$, and
\item every other exceptional divisor $F$ of $g$ has discrepancy $-1$. \qed
\end{enumerate}
\end{prop}

\begin{lem}\label{mmp.on.crep.D.lem}
  Let $f:Y\to Z$ be a projective morphism between normal spaces. Assume that
  $(Y, \Delta)$ is $\q$-factorial, dlt and $K_Y+\Delta\sim_{f,\q}0$.  Let $D$ be an
  effective $\z$-divisor such that $ \supp D\subset \supp \Delta$.  Then the
  $(-D)$-MMP runs and terminates in a good minimal model if $D$ does not dominate
  $Z$.
\end{lem}

\begin{proof} The $(-D)$-MMP is the same as the $(-\epsilon D)$-MMP, which in turn
  agrees with the $(K_Y+\Delta -\epsilon D)$-MMP since $K_Y+\Delta\sim_{f,\q}0$.

  If $f$ is birational, the $(K_Y+\Delta -\epsilon D)$-MMP runs and terminates by
  \cite{birkar11, MR3032329}.

  If $f$ is not birational, then the generic fiber of $(Y,\Delta -\epsilon D)\to X$
  is the same as the generic fiber of $(Y,\Delta)\to X$, and the latter is a good
  minimal model by assumption. Thus the MMP for $(Y,\Delta -\epsilon D)\to X$ runs
  and terminates by \cite{MR3032329}.

  The above references work for varieties; see \cite{dvp, k-nqmmp} for algebraic
  spaces of finite type, \cite{fujino2022minimal} for analytic spaces and
  \cite{lyu-mur} for the most general settings.
\end{proof}

\subsection*{Crepant log structures}{\ }

\noindent
For the general case of Theorem~\ref{DB.epsilon.thm}, we first study what the above
proof gives. Keeping in mind the inductive arguments of \cite{k-db}, we do this for
crepant log structures. The end result is Lemma~\ref{DB.epsilon.crep.lem}. Then
induction and repeated use of (\ref{db.sings.props.say}.2) completes the proof in
Proposition~\ref{W.db.ind.lem}.

\begin{defn}\label{crep.l.s.def}
  A  {\it crepant log structure} is a dominant, projective morphism with connected fibers
  $g:(Y, \Delta)\to Z$,  where  $(Y, \Delta)$ is lc, $Z$ is normal  and
  $K_Y+ \Delta\sim_{g, \r}0$.

  If $(X, \Delta)$ is lc, then the identity
  $(X, \Delta)\to X$ is a crepant log structure.

  As a generalization of (\ref{db.sings.props.say}.1), $Z$ is Du~Bois
  \cite[6.31]{kk-singbook}.

  For an irreducible $V\subset Z$ we define $\mld(V,Y,\Delta)$ as the infimum of the
  numbers $1+a(E,Y,\Delta)$ where $E$ runs through all divisors over $Y$ that
  dominate $V$.

  As in Definition~\ref{crep.log.str.mld.defn}, if $V\subset Z$ is a closed subset
  with irreducible components $V_i$, then we set
  $\mld(V,Y,\Delta):=\max_i \bigl\{\mld(V_i,Y,\Delta)\bigr\}$.

  We will use the following property proved in \cite{k-1/6}, see also
  \cite[7.5]{kk-singbook}.
\begin{equation}
  \label{eq:2}
  \mld(V_1\cap V_2,Y,\Delta)\leq \mld(V_1,Y,\Delta)+\mld(V_2,Y,\Delta).
\end{equation}
\end{defn}

The following generalization of Theorem~\ref{DB.epsilon.thm} is better suited for
induction.

\begin{thm}\label{DB.epsilon.crep.thm}
  Let $g:(Y, \Delta)\to Z$ be a crepant log structure. Set $n=\dim Y$ and let
  $V\subset Z$ be a closed subset such that $\mld(V, Y,\Delta)<\lcg(n)$.  Then $V$
  has Du~Bois singularities.
\end{thm}

\noindent
Next we see what the method of (\ref{DB.epsilon.thm.klt.pf}) gives for crepant log
structures.

\begin{notation}\label{nota.not.1}
  Let $g:(Y, \Delta)\to Z$ be a crepant log structure.  For a closed subset
  $Z_1\subset Z$, let $Z_1^{\diamond}\subset Z_1$ denote the union of those log
  canonical centers of $(Y,\Delta)$ that are contained in $Z_1$, but are nowhere
  dense in it.

  Note that $Z^{\diamond}$ is the non-klt locus of $g:(Y, \Delta)\to Z$.

  If $Z_1$ itself is the union of log canonical centers of $(Y,\Delta)$, then $Z_1$
  is seminormal and $Z_1\setminus Z_1^{\diamond}$ is normal by \cite{ambro, fuj-book}
  and \cite[4.32]{kk-singbook}.
\end{notation}

 \begin{lem} \label{DB.epsilon.crep.lem}
  Let $g:(Y, \Delta)\to Z$ be a crepant log structure with klt generic fiber. Set    $n=\dim Y$ and let $V\subset Z$ be a closed subset such that   $\mld(V, Y,\Delta)<\lcg(n)$.
Then $V\cup Z^{\diamond}$  has Du~Bois singularities.
\end{lem}

\begin{proof} By Proposition~\ref{thm:one-div-models}, we may assume that
  $(Y, \Delta)$ is $\q$-factorial, dlt, and there is a divisor $ D=\sum D_i\subset Y$
  such that $\rdown{\Delta}\subset D$, $\mld(D, Y, \Delta)<\lcg(n)$, and
  $g(D)=V\cup Z^{\diamond}$.  After running the MMP for $K_Y+ \Delta-\epsilon D$ for
  some $\epsilon>0$ as in Lemma~\ref{mmp.on.crep.D.lem}, we may also assume that $-D$
  is $g$-semiaple. As in (\ref{1/6.sn.lem}.1),
  \[
    -D\sim_{g,\r} K_Y+(\Delta-\epsilon D)+(1-\epsilon)(-D),
    \eqno{(\ref{DB.epsilon.crep.lem}.1)}
  \]
  where $(Y, \Delta-\epsilon D)$ is klt and $(1-\epsilon)(-D)$ is $g$-semiample.
  Also note that $D=\supp g^{-1}(V)$, since $-D$ is $g$-nef and the fibers are
  connected. We can now use Theorem~\ref{1/6.sn.lem} to get that
  \[
    \o_Z\bigl(-(V\cup Z^{\diamond})\bigr)\to R g_*\o_Y(-D)
    \eqno{(\ref{DB.epsilon.crep.lem}.2)}
  \]
  has a left inverse.  As we noted in Definition~\ref{crep.l.s.def}, $Z$ is Du~Bois.
  By (\ref{db.sings.props.say}.3) these imply that $V\cup Z^{\diamond}$ is Du~Bois.
\end{proof}

 \begin{cor}\label{W+lc.db.lem}
   Let $g:(Y,\Delta)\to X$ be a crepant log structure of dimension $n$, and
   $Z\subset X$ a union of some of its log canonical centers.  Let $V\subset Z$ be a
   closed subset such that $\mld(V, Y,\Delta)<\lcg(n)$. Then $V\cup Z^{\diamond}$ is
   Du~Bois.
\end{cor}

\begin{proof} We may assume that $(Y,\Delta)$ is $\q$-factorial and dlt.

  For each irreducible component $Z_i\subset Z$, let $Y_i\subset Y$ be a minimal
  dimensional log canonical center of $(Y,\Delta)$ that dominates $Z_i$. Set
  $\Theta_i:=\diff^*_{Y_i}\Delta$ as in \cite[4.18.4]{kk-singbook}.
   
  Let $\pi_i:\bar Z_i\to Z_i$ denote the normalization. Stein factorization of
  $Y_i\to Z_i$ gives $g_i: Y_i\to \widetilde Z_i$ and
  $\tau_i:\widetilde Z_i\to \bar Z_i$.  The
  $g_i: \bigl(Y_i, \Theta_i\bigr)\to \widetilde Z_i$ are crepant log structures with
  klt general fibers.

  Precise inversion of adjunction \cite[7.10]{kk-singbook} shows that
  $(\pi_i\circ \tau_i)^{-1}(Z^{\diamond})=(\widetilde Z_i)^{\diamond}$ and
  $\mld\bigl(\widetilde V_i, Y_i, \Theta_i\bigr)\leq \mld(V, Y,\Delta)$, where
  $\widetilde V_i:=(\pi_i\circ \tau_i)^{-1}(V)$.

  The $\widetilde Z_i$ are Du~Bois by (\ref{db.sings.props.say}.1), and the
  $\widetilde V_i\cup\widetilde Z_i^{\diamond}$ are Du~Bois by
  Lemma~\ref{DB.epsilon.crep.lem}.

  Set $\bar V_i:=\pi_i^{-1}(V)$, $\bar Z_i^{\diamond}:=\pi_i^{-1}(Z^{\diamond})$ and
  $\bar Z^{\diamond}:=\cup_i \bar Z_i^{\diamond}$.
  
  The normalized trace map splits $\o_{\bar Z_i}\into (\tau_i)_*\o_{\widetilde Z_i}$,
  and hence also splits
  \[
    \o_{\bar Z_i}\bigl(-(\bar V_i\cup\bar Z_i^{\diamond})\bigr)\into
    (\tau_i)_*\o_{\widetilde Z_i}\bigl(-(\widetilde V_i\cup\widetilde
    Z_i^{\diamond})\bigr).
  \]
  Using the first splitting and applying (\ref{db.sings.props.say}.3) to
  $\bigl(\widetilde Z_i, \emptyset\bigr)\to \bigl(\bar Z_i, \emptyset\bigr)$ shows
  that $\bar Z_i$ is Du~Bois.  Applying (\ref{db.sings.props.say}.3) and the second
  splitting now gives that $\bar V_i\cup \bar Z_i^{\diamond}$ is Du~Bois.

  As we noted in (\ref{nota.not.1}), $Z$ is seminormal and normal outside
  $Z^{\diamond}$, thus $\o_Z(-Z^{\diamond})=\pi_*\o_{\bar Z}(-\bar Z^{\diamond})$,
  hence
  \[
    \o_Z\bigl(-(V\cup Z^{\diamond})\bigr)=\pi_*\o_{\bar Z}\bigl(-(\bar V\cup\bar
    Z^{\diamond})\bigr).
  \]
  Since $Z^{\diamond}$ is Du~Bois by induction on the dimension, the first splitting
  shows that $Z$ is Du~Bois. Using (\ref{db.sings.props.say}.3) and the second
  splitting gives that $V\cup Z^{\diamond}$ is Du~Bois.
\end{proof}

\subsection*{Proof of Theorems~\ref{DB.epsilon.thm} and \ref{DB.epsilon.crep.thm}}{\
}

\noindent
Since Theorem~\ref{DB.epsilon.crep.thm} implies Theorem~\ref{DB.epsilon.thm},
all that remains is to formulate a variant of Theorem~\ref{DB.epsilon.crep.thm}
that allows for  induction on the dimension.
The strongest version would  use the language of quasi-log structures as in \cite{fuj-book}. They appear  implicitly in 
the proof of  Proposition~\ref{W.db.ind.lem}, but our approach works well enough.

Note that Theorem~\ref{DB.epsilon.crep.thm} is the $Z=X$ special case of
Proposition~\ref{W.db.ind.lem}. Thus the proof of Proposition~\ref{W.db.ind.lem}
yields Theorem~\ref{DB.epsilon.crep.thm}.

\begin{prop}\label{W.db.ind.lem}    Let  $g:(Y,\Delta)\to X$ be a  crepant log structure  of dimension $n$, and $Z\subset X$  a union of some of its log canonical centers, allowing $Z=X$.  Let
  $V\subset Z$ be a closed subset such that $\mld(V, Y,\Delta)<\lcg(n)$. Then
  $V$ is Du~Bois.
\end{prop}

\begin{proof} The proof is by induction on $\dim Z$. If $\dim Z=0$ then $V$ is a
  union of smooth points, hence Du~Bois.
  
  Write $V=V_1\cup V_2$, where $V_2\subset Z^{\diamond}$ and none of the irreducible
  components of $V_1$ is contained in $ Z^{\diamond}$.

  Note that $V_1\cup Z^{\diamond}$ is Du~Bois by (\ref{W+lc.db.lem}), and so is
  $Z^{\diamond}$.  Furthermore, $\mld(V_1\cap Z^{\diamond}, X,\Delta)<\lcg(n)$ by
  (\ref{eq:2}), hence $V_1\cap Z^{\diamond} $ is Du~Bois by induction since
  $\dim Z^{\diamond}<\dim Z$.  Thus $V_1$ is Du~Bois by (\ref{db.sings.props.say}.2).

  Next $\mld(V\cap Z^{\diamond}, X,\Delta)<\lcg(n)$ by (\ref{eq:2}), hence
  $V\cap Z^{\diamond} $ is Du~Bois by induction.  We already checked that $V_1$ and
  $V_1\cap Z^{\diamond} = V_1\cap (V\cap Z^{\diamond}) $ are Du~Bois.  Thus
  $V=V_1\cup (V\cap Z^{\diamond}) $ is Du~Bois by (\ref{db.sings.props.say}.2).
\end{proof}

\subsection*{Conjectures and comments}{\ }

\noindent
In the proof of Theorem~\ref{DB.epsilon.thm}, instead of $\mld(V, Y,\Delta)<\lcg(n)$,
we only use the assumption that $\mld(V_i, Y,\Delta)<\lcg(n-1)$ if $V_i$ is contained
in a log canonical center, and $\mld(V_i, Y,\Delta)<\lcg(n)$ otherwise. This suggests
that the following should be true.

\begin{ques}\label{DB.epsilon.crep.ques}
  Let $g:(Y, \Delta)\to Z$ be a crepant log structure. Let $V\subset Z$ be a closed
  subset with irreducible components $V_i$.  Let $Z_i\supset V_i$ be the minimal log
  canonical center that contains $V_i$ (we allow $Z_i=X$).  Assume that
  $\mld(V_i, Y,\Delta)<\lcg\bigl(\dim(Z_i)\bigr)$ for every $i$.  Is $V$ necessarily
  Du~Bois?
\end{ques}

\noindent
A related question is the following.

\begin{conj}\label{near.lc.=.lc.q}
  Let $(X, \Delta)$ be a quasi-projective, log canonical pair of dimension $n$.  Let
  $V\subset X$ be a closed subset such that $\mld(V, X,\Delta)<\lcg(n)$ and $V$
  contains all log canonical centers of $(X, \Delta)$.

  Then there is a log canonical pair $(X, \Theta)$ such that $V$ is the union of all
  log canonical centers of $(X, \Theta)$.
\end{conj}
 
Note that usually one can not choose $\Theta\geq \Delta$, as shown by the
2-dimensional example $\bigl(\a^2, (1-\eta)(x=0)+(1-\eta)(y=0)+\eta(x=y)\bigr)$.
Also, if $Z$ is a log canonical center of $(X, 0)$, then it is also a log canonical
center of any $(X, \Theta)$, so the assumption that $V$ contain all log canonical
centers of $(X, \Delta)$ is necessary in many cases.

\medskip

If $(X, \Delta)$ is klt, then a proof of Conjecture~\ref{near.lc.=.lc.q} is given in \cite{k-dano}. Together with \cite{k-db, k-db2}, this gives another proof of
the klt case of Theorem~\ref{DB.epsilon.thm}.
However, even the full conjecture does not seem to imply 
Theorem~\ref{DB.epsilon.thm}, since we get no information about
those $V_i$ that are contained in a log canonical center of $(X, \Delta)$.

A positive answer to the following stronger version would imply
Theorem~\ref{DB.epsilon.thm}.

\begin{ques}\label{near.lc.=.lc.q.2}
  Let $(X, \Delta)$ be a quasi-projective,  log canonical pair of dimension $n$.
  Is there  a  log canonical  pair
  $(X, \Theta)$ such that every 
  irreducible subvariety satisfying $\mld(V, X,\Delta)<\lcg(n)$
   is a  log canonical center of  $(X, \Theta)$?
\end{ques}

Note that usually we can not achieve that the  log canonical centers  are exactly the
$\{V\colon \mld(V, X,\Delta)<\lcg(n)\}$. Indeed, any intersection of log canonical centers is a union of log canonical centers, but this does not hold for the 
$\mld(V, X,\Delta)<\lcg(n)$ condition.

\begin{ack}
  We thank S.~Filipazzi for corrections, and O.~Fujino for pointing out that the
  arguments do not cover the analytic case of Theorem 1, despite our earlier claim.
  Partial financial support to JK was provided by the NSF under grant number
  DMS-1901855. SK was supported in part by NSF Grants DMS-1951376 and DMS-2100389,
  and a Simons Fellowship (Award Number 916188).
\end{ack}


\def\cprime{$'$} \def\cprime{$'$} \def\cprime{$'$} \def\cprime{$'$}
  \def\cprime{$'$} \def\dbar{\leavevmode\hbox to 0pt{\hskip.2ex
  \accent"16\hss}d} \def\cprime{$'$} \def\cprime{$'$}
  \def\polhk#1{\setbox0=\hbox{#1}{\ooalign{\hidewidth
  \lower1.5ex\hbox{`}\hidewidth\crcr\unhbox0}}} \def\cprime{$'$}
  \def\cprime{$'$} \def\cprime{$'$} \def\cprime{$'$}
  \def\polhk#1{\setbox0=\hbox{#1}{\ooalign{\hidewidth
  \lower1.5ex\hbox{`}\hidewidth\crcr\unhbox0}}} \def\cdprime{$''$}
  \def\cprime{$'$} \def\cprime{$'$} \def\cprime{$'$} \def\cprime{$'$}
\providecommand{\bysame}{\leavevmode\hbox to3em{\hrulefill}\thinspace}
\providecommand{\MR}{\relax\ifhmode\unskip\space\fi MR }
\providecommand{\MRhref}[2]{%
  \href{http://www.ams.org/mathscinet-getitem?mr=#1}{#2}
}
\providecommand{\href}[2]{#2}

\end{document}